\documentclass[11pt]{article}
\usepackage{amsmath,amsthm,amssymb,hyperref,mathtools,bbm}
\usepackage[left=1in,top=1in,right=1in]{geometry}

\theoremstyle{definition}
\newtheorem{theorem}{Theorem}[section]
\newtheorem{lemma}[theorem]{Lemma}
\newtheorem{fact}[theorem]{Fact}
\newtheorem{proposition}[theorem]{Proposition}

\newtheorem{problem}[theorem]{Problem}
\newtheorem{question}[theorem]{Question}
\newtheorem{claim}[theorem]{Claim}
\counterwithin{equation}{section}
\allowdisplaybreaks

\newcommand{\E}{\mathbb{E}}

\newcommand{\Bin}{\mathrm{Bin}}
\newcommand{\Var}{\mathrm{Var}}
\newcommand{\X}{\mathbf{X}}

\renewcommand{\S}{S}
\newcommand{\R}{\mathcal{R}}
\newcommand{\N}{\mathcal{N}}
\newcommand{\A}{\mathcal{A}}
\newcommand{\B}{\mathcal{B}}
\newcommand{\U}{\mathcal{U}}
\newcommand{\rmspan}{\mathrm{span}}

\newcommand{\1}{\mathbf{1}}
\newcommand{\bone}{\mathbbm{1}}

\newcommand{\lrpar}[1]{\left( #1 \right)}
\newcommand{\lrabs}[1]{\left| #1 \right|}
\newcommand{\norm}[1]{\left|\left| #1 \right|\right|_2}

\newcommand{\nul}{\mathrm{Null}}
\renewcommand{\r}{r}

\newcommand{\del}{\delta}
\renewcommand{\c}[1]{\mathcal{#1}}
\newcommand{\sm}{\setminus}
\newcommand{\rk}{\mathrm{rank}}
\renewcommand{\a}{a}
\renewcommand{\b}{b}
\newcommand{\supp}{\mathrm{supp}}
\newcommand{\sub}{\subseteq}
\newcommand{\half}{\frac{1}{2}}

\title{Semi-restricted Rock, Paper, Scissors}
\date{}

\author{Sam Spiro\footnote{Department of Mathematics, University of California at San Diego, La Jolla, CA, 92093 USA. Email: {\tt sspiro@ucsd.edu}, {\tt esurya@ucsd.edu}, {\tt jzeng@ucsd.edu}.}\ \textsuperscript{,}\footnote{Supported by the National Science Foundation Graduate Research Fellowship Grant DGE-1650112.} \and
	Erlang Surya\footnotemark[1]\ \textsuperscript{,}\footnote{Supported by the National Science Foundation Grant DMS-2225631.} \and
	Ji Zeng\footnotemark[1]\ \textsuperscript{,}\footnote{Supported by the National Science Foundation Grant DMS-1800746.}}

\begin{document}
	\maketitle
	
	\begin{abstract}
		Consider the following variant of Rock, Paper, Scissors (RPS) played by two players Rei and Norman.  The game consists of $3n$ rounds of RPS, with the twist being that Rei (the restricted player) must use each of Rock, Paper, and Scissors exactly $n$ times during the $3n$ rounds, while Norman is allowed to play normally without any restrictions.  Answering a question of Spiro, we show that a certain greedy strategy is the unique optimal strategy for Rei in this game, and that Norman's expected score is $\Theta(\sqrt{n})$. Moreover, we study semi-restricted versions of  general zero sum games and prove a number of results concerning their optimal strategies and expected scores, which in particular implies our results for semi-restricted RPS.
	\end{abstract}
	
	\section{Introduction}
	\subsection{Rock, Paper, Scissors}
	The game Rock, Paper, Scissors, or RPS for short, is a popular game whose first known usage dates back to China nearly 2000 years ago.  A round of RPS consists of two players simultaneously selecting to play either Rock, Paper, or Scissors; where Rock beats Scissors, Scissors beats Paper, and Paper beats Rock.  If a player uses a move which beats their opponents move, then that player gains a point and their opponent loses a point, with the match resulting in a draw if both players select the same move.
	
	RPS, while fun to play, is not particularly interesting from a mathematical perspective. As a symmetric zero-sum game, both players have an expected score of $0$, and it is easy to show that the unique optimal strategy for both players is to choose each of the three options with uniform probability; see for example \cite[Chapter 17]{von2007theory}.  In this paper we study a non-trivial variant of RPS, inspired by similar games introduced by Fukumoto~\cite{fukumotokaiji} and Spiro~\cite{spiro2022online}, called \textit{semi-restricted RPS}.
	
	Semi-restricted RPS is a perfect-information zero-sum game played by two players named Rei and Norman. The game consists of $3n$ rounds of RPS, with the twist being that Rei must use each of Rock, Paper, and Scissors exactly $n$ times during the $3n$ rounds, while Norman is allowed to play normally without any restrictions. It is clear that Norman has an advantage in this game, since in particular he is guaranteed to win the last round of the game (assuming he is paying attention and has a good memory). However, it is unclear how much more Norman is expected to win over Rei when both play optimally, and it is also unclear what the optimal strategies are for either player.
	
	One simple strategy that Rei can implement is what we call the \textit{greedy strategy}.  Under this strategy, if Rei can still play each of Rock, Paper, and Scissors, she selects each option with probability 1/3, regardless of how many actions remain of each option.  If she can only play, say, Rock and Paper, then she chooses Paper with probability 2/3 and Rock with probability 1/3.  More generally, if she has two options remaining, she will choose the option which beats the other with probability 2/3 and the other option with probability 1/3.  And of course, if only one option remains, she plays this with probability 1.
	
	This strategy is ``greedy'' because with this strategy, Rei minimizes her expected loss for any given round.  Indeed, if she has all three options remaining, then playing each with probability 1/3 makes it so that, regardless of what Norman does, Rei is just as likely to win as she is to lose.  Similarly if only Rock and Paper remains, then Norman should only ever play Paper or Scissors (since Rock is guaranteed not to win). If Rei follows the greedy strategy and Norman plays Paper, then Rei will lose with probability 1/3.  If Norman plays Scissors, then Rei will lose with probability 2/3 but win with probability 1/3.  Thus regardless of what Norman does, Rei will expect to lose 1/3 points each round under the greedy strategy, and one can show that this is best possible.
	
	While the greedy strategy is optimal if Rei is only concerned about a given round, it is far from clear that this is a good strategy overall.  Indeed, say the game reaches the point where Rei can play 100 Rocks, 100 Papers, and just 1 Scissors.  Intuitively, in this scenario Rei should not play Scissors with high probability, as doing so will severely limit her remaining options for the rest of the game.  As such the following may come as a bit of a surprise.
	\begin{theorem}\label{thm:RPS}
		The greedy strategy is the unique optimal strategy for Rei in semi-restricted RPS.  Moreover, if the game consists of $3n$ rounds with both players playing optimally, then Norman's expected score is $\Theta(\sqrt{n})$.
	\end{theorem}
	Here and throughout the paper we make use of standard asymptotic notation, see Subsection~\ref{sub:org} for precise definitions.

	\subsection{A More General Setting}
	The problem of studying semi-restricted RPS was first proposed by Spiro~\cite{spiro2022online}. In~\cite{spiro2022online}, a semi-restricted version of another classical zero-sum game, Matching Pennies, was introduced.  This semi-restricted game was motivated by a certain card guessing game studied by Diaconis and Graham \cite{diaconis1981analysis} which has recently received a fair amount of attention, see for example \cite{alimohammadi2021sequential,diaconis2022card,he2021card,krityakierne2022no,liu2021card}.  
	
	While it is easy to generalize these examples to study ``semi-restricted versions'' of arbitrary simultaneous zero-sum games, for simplicity we will focus on games which come from digraphs. However, most of our results hold in broader generality.
	Recall that a \textit{digraph} $D$ consists of a set of vertices $V(D)$ together with a set of ordered pairs of vertices $E(D)$ called \textit{arcs}. We will sometimes denote arcs $(u,v)$ as $u\to v$ or $uv$.  Throughout this paper we will only consider digraphs without self-loops and which have at most one arc between a given pair of vertices.
	
	Given a digraph $D$, we define the \textit{$D$-game} by having two players simultaneously select a vertex of $D$.  If the selected vertices are $u,v$ and $uv\in E(D)$, then the player who chose $u$ gains a point and the player who chose $v$ loses a point, and nothing happens if $uv,vu\notin E(D)$.  For example, if $D$ is the circuit of length 3, i.e.\ the 3-vertex digraph with arcs $1\to 2\to 3\to 1$, then the $D$-game is equivalent to RPS. Observe that the $D$-game is always a symmetric zero-sum game.
	
	We will say that a vector $\r$ is a \textit{restriction vector} (with respect to a digraph $D$) if it is a vector of non-negative integers indexed by $V(D)$.  Given a digraph $D$ and restriction vector $\r$, we define the \textit{semi-restricted $D$-game with parameter $\r$} by having two players Rei and Norman iteratively play the $D$-game for a total of $\sum_u \r_u$ rounds with the restriction that Rei must select each $u\in V(D)$ exactly $\r_u$ times.  We let\footnote{We will define these terms more formally in Section~\ref{sec:optimal}.} $\S_D(\r)$ denote the expected score for Norman when both players play optimally in the semi-restricted $D$ game with parameter $\r$, and throughout we let $\1$ denote the all 1's vector of dimension $|V(D)|$.  For example, Theorem~\ref{thm:RPS} says $\S_D(n\cdot \1)=\Theta(\sqrt{n})$ when $D$ is the circuit of length 3.
	
	Determining optimal strategies for semi-restricted $D$-games in full generality seems impossible.  Nevertheless, we are able to obtain effective bounds on $\S_D(\r)$ for all $D$.  To state these results, given a digraph $D$ and $v\in V(D)$, we define the \textit{out-neighborhood} $N^+(v)=\{u:(v,u)\in E(D)\}$, and similarly we define the \textit{in-neighborhood} $N^-(v)=\{u:(u,v)\in E(D)\}$.  We let $d^+(v)=|N^+(v)|$ and $d^-(v)=|N^-(v)|$ denote the \textit{out-degree} and \textit{in-degree} of $v$, respectively. 
	
	A basic observation is that for all vertices $v\in V(D)$, we have
	\[\S_D(n\cdot \1)\ge (d^+(v)-d^-(v))n.\]
	Indeed, if Norman uses the deterministic strategy of playing $v$ every round, then he will win exactly $d^+(v)n$ rounds and lose exactly $d^-(v)n$ rounds, so he can achieve an expected score of at least $(d^+(v)-d^-(v))n$ with this strategy.  It turns out that this trivial lower bound is close to best possible.
	\begin{theorem}\label{thm:UpperUniform}
		For all digraphs $D$ and $n\ge 1$, we have
		\[\max_v\{d^+(v)-d^-(v)\}n\le \S_D(n\cdot \1)\le \max_v\{d^+(v)-d^-(v)\}n+O_D(\sqrt{n}).\]
	\end{theorem}
	Similar bounds hold for general restriction vectors $\r$, though in this case our bound on the error term is weaker.
	\begin{theorem}\label{thm:UpperNonuniform}
		For all digraphs $D$ and restriction vectors $\r$, we have
		\[ \max_v\left\{\sum_{u\in N^+(v)}\r_u-\sum_{u\in N^-(v)}\r_u\right\}\le \S_D(\r)\le \max_v\left\{\sum_{u\in N^+(v)}\r_u-\sum_{u\in N^-(v)}\r_u\right\}+O_D(M^{2/3}),\]
		where $M=\max_v \r_v$.
	\end{theorem}
	
	The bounds of Theorem~\ref{thm:UpperUniform} are asymptotically tight except when $d^+(v)=d^-(v)$ for all $v$, and when this happens we say that $D$ is \textit{Eulerian}. A class of Eulerian digraphs that are of particular interest to us are Eulerian tournaments, where  a \textit{tournament} is a digraph such that either $uv\in E(D)$ or $vu\in E(D)$ is in $D$ for all distinct $u,v\in E(D)$ (equivalently, tournaments are orientations of complete graphs).  For example, the circuit of length 3 is an Eulerian tournament.  More generally, any reasonable extension of RPS (such as ``Rock, Paper, Scissors, Lizard, Spock'' or the infamous ``RPS-25'' consisting of 25 different options) will have a corresponding digraph which is an Eulerian tournament.
	
	By Theorem~\ref{thm:RPS}, we know that semi-restricted Rock, Paper, Scissors has a simple greedy strategy for Rei which is optimal, and given this, it is not too hard to determine the expected score when both players play optimally.  Unfortunately, it turns out that Rei does not have a ``simple'' optimal strategy in the semi-restricted $D$-game for almost every Eulerian tournament $D$, see Theorem~\ref{thm:randomEulerian} for a precise statement. Despite this significant hurdle, we can still determine the expected score when both players play optimally.
	\begin{theorem}\label{eulerian}
		If $D$ is a non-empty Eulerian tournament, then
		\[\S_D(n\cdot \1)=\Theta_D(\sqrt{n}).\]
	\end{theorem}
	In fact, we will show more generally that $\S_D(n\cdot \1)=\Theta_D(\sqrt{n})$ whenever $D$ is an Eulerian digraph satisfying certain spectral conditions, see Theorem~\ref{lowergood}.
	
	Lastly, it is natural to consider not just semi-restricted games where one player is restricted, but also games where both players are restricted.  Indeed, several games of this form appeared in the manga Tobaku Mokushiroku Kaiji~\cite{fukumotokaiji}, and one such game was studied by Spiro~\cite{spiro2022online}.  It turns out that these games are much simpler to analyze.  In particular, in every such game, it is optimal for both players to uniformly at random choose from their multiset of available actions each round, see Appendix~\ref{sec:restricted} for more on this.
	
	\subsection{Organization and Notation}\label{sub:org}
	
	The remainder of this paper is organized as follows.  In Section~\ref{sec:optimal} we provide a more formal definition of $\S_D(\r)$ and determine the optimal strategies for every semi-restricted $D$-game when $D$ has at most 3 vertices, in particular proving the first half of Theorem~\ref{thm:RPS}. In Section~\ref{sec:analysis} we prove general bounds on $\S_D(\r)$ and prove Theorems~\ref{thm:UpperUniform}, \ref{thm:UpperNonuniform}, \ref{eulerian} and the second half of Theorem~\ref{thm:RPS}. In Section~\ref{sec:oblivious} we show that almost every (Eulerian) tournament does not have a ``simple'' optimal strategy for Rei. We end with some open problems in Section~\ref{sec:conclusion}.
	
	We recall the following standard asymptotic notation.  For two functions $f,g$ depending on $n$, we write $g=O(f)$ to mean there exists some $C>0$ such that $g(n)\le C f(n)$ for all $n$.  Similarly $g=\Omega(f)$ indicates $g(n)\ge c f(n)$ for some $c>0$, and $g=\Theta(f)$ means $g=O(f)$ and $g=\Omega(f)$.  We write, for example, $g=O_D(f)$ to indicate that there is some $C>0$ depending on the parameter $D$ such that $g(n)\le C f(n)$ for all $n$.  We write $g=o(f)$ to indicate $\lim_{n\to \infty} \frac{g(n)}{f(n)}=0$.  
	
	Given a vector $\r$, we define the \textit{support} $\supp(\r)$ to be the set of $v$ with $\r_v>0$.  Similarly given a random variable $X$, we define its \textit{support} $\supp(X)$ to be the set of $x$ with $\Pr(X=x)>0$.  Given an event $A$ we let $\bone(A)$ denote the indicator function which is 1 if $A$ occurs and 0 otherwise.
	
	\section{Optimal Strategies}\label{sec:optimal}
	In this section we prove several results about optimal strategies.    We begin by establishing formal definitions for some of the terms that were informally defined in the introduction.
	
	Given a digraph $D$, a \textit{strategy for Norman} in the semi-restricted $D$-game is a function $\N$ from restriction vectors to random variables such that $\supp(\N(\r))\subseteq V(D)$ for all $\r$.  A \textit{strategy for Rei} is a function $\R$ from restriction vectors to random variables such that $\supp(\R(\r))\subseteq \supp(\r)$ for all $\r$.  Informally, this simply says that Norman is allowed to play any vertex of $D$, while Rei can only play vertices in $\supp(\r)$.
	
	We let $\del_v$ be the vector which has a 1 in the position corresponding to $v\in V(D)$ and 0's everywhere else.  Given strategies $\N,\R$ for Norman and Rei respectively and a restriction vector $\r$, we define the \textit{score} (for Norman) $\mathbf{S}_D(\r;\N,\R)$ recursively by setting $\mathbf{S}_D(\r;\N,\R)=0$ if $\r=(0,\ldots,0)$, and otherwise having
	\[
	\mathbf{S}_D(\r;\N,\R)=\begin{cases}
		1+\mathbf{S}_D(\r-\del_{\R(\r)};\N,\R)& \text{if }\R(\r)\in N^+(\N(\r)),\\
		-1+\mathbf{S}_D(\r-\del_{\R(\r)};\N,\R)& \text{if }\R(\r)\in N^-(\N(\r)),\\
		\mathbf{S}_D(\r-\del_{\R(\r)};\N,\R)& \text{otherwise}.
	\end{cases}
	\]
	Informally, this says Norman and Rei play a round of the semi-restricted $D$-game with parameter $\r$, record the change in score based on whether Norman's choice beats Rei's or not, and then the game continues with one action of $\R(\r)$ removed for Rei. Note that $\mathbf{S}_D(\r;\N,\R)$ is a random variable. We define \textit{the expected score} $S_D(\r;\N,\R):=\E[\mathbf{S}_D(\r;\N,\R)]$.
	
	For $\r$ a restriction vector and $\R$ a Rei's strategy, we define
	\[
	\S_D(\r;\R)=\max_{\N} \S_D(\r;\N,\R)\text{\qquad and\qquad} \S_D(\r)=\min_{\R}\S_D(\r;\R),
	\]where the maximum and minimums run through all possible strategies of Norman and Rei as appropriate. The fact that these maximums and minimums exist can be seen, for instance, by observing that the space of possible strategies forms a compact set. Intuitively, $\S_D(\r;\R)$ is Norman's score if he plays optimally and Rei uses strategy $\R$ and has restriction vector $\r$. If $\R$ is such that $\S_D(\r)=\S_D(\r;\R)$ for all vectors $\r$, then we say that $\R$ is an \textit{optimal strategy} for Rei.
	
	We start with a basic fact about the behavior of $\S_D$.
	\begin{fact}\label{fact:recurrence}
		Let $\R$ be a strategy for Rei. If $\r$ is a restriction vector and $p_u=\Pr(\R(\r)=u)$ for each $u\in V(D)$, then
		\[\S_D(\r;\R)=\max_v\left\{\sum_{u\in N^+(v)}p_u-\sum_{u\in N^-(v)}p_u\right\}+\sum_{u} p_u \S_D(\r-\del_u;\R).\]
	\end{fact}
	Intuitively, this result says that if Rei is playing according to strategy $\c{R}$, then Norman should always play the vertex $v$ which maximizes his expected score for each round, i.e.\ the $v$ such that $\sum_{u\in N^+(v)}p_u-\sum_{u\in N^-(v)}p_u$ is as large as possible.  A close analog of this is proven in Spiro~\cite[Lemma 2.1]{spiro2022online}. One can easily give a formal proof of this results by adjusting the corresponding proof in~\cite{spiro2022online}.

	The key lemma we need for this section is the following. Roughly speaking, it says that in any semi-restricted $D$-game, the optimal expected score will not change dramatically if among all the remaining options of Rei, one action of $u$ is changed into one action of $v$.
	\begin{lemma}\label{lem:switch}
		Given $u,v\in V(D)$ and a restriction vector $\r$ with $\r_u,\r_v\geq 1$, we have
		\begin{equation}\label{eq:switch}
			\S_D(\r-\del_u)\le \S_D(\r-\del_v)+\alpha(u,v),
		\end{equation}where \[\alpha(u,v):=\begin{cases}
			2&\text{if $N^+(u)\cap N^-(v)\neq \emptyset$,}\\
			0&\text{if $N^+(u)\cup N^-(v)=\emptyset$,}\\
			1&\text{otherwise.}\\
		\end{cases}\]Moreover, we have strict inequality in \eqref{eq:switch} when either $N^-(u)\neq \emptyset$ or $\alpha(u,v)=2$.
	\end{lemma}
	\begin{proof}
		We first assume the condition $N^-(u)\neq \emptyset$ and prove the strict inequality
		$$\S_D(\r-\del_u)<\S_D(\r-\del_v)+\alpha(u,v).$$
		After its proof we briefly discuss the other scenarios of this theorem, i.e.\ when $\alpha(u,v)=2$ or when we conclude non-strict inequalities.
		
		Suppose for contradiction that there is some $\r$ with $\S_D(\r-\del_u)\geq \S_D(\r-\del_v)+\alpha(u,v)$, and choose this $\r$ with $\sum \r_w$ as small as possible. Let $\R$ be an optimal strategy for Rei, and let $p_w=\Pr(\R(\r-\del_v)=w)$. The idea now is for Rei to consider a different (and possibly suboptimal strategy) $\R'$ which has Rei acting as if her restriction vector is $\r-\del_v$ whenever it is $\r-\del_u$.  That is, we would like to choose $\R'$ such that $\Pr[\R'(r-\del_u)=w]=p_w$ for all $w$.  However, this can cause issues when $\r_u=1$, since in this case $\R'$ may choose $u$ with positive probability despite $\r-\del_u$ having no copies of $u$ available for her to play.  We get around this by having $\R'$ select $v$ whenever $\R$ would have selected $u$.  More precisely, we define a strategy $\R'$ by having $\R'(\r')=\R(\r')$ if $\r'\neq \r-\del_u$, and otherwise setting
		
		$$ p'_w:=\Pr(\R'(\r-\del_u)=w)=\begin{cases} p_u+p_v & w=v,\\ 0 & w=u,\\ p_w & w\ne u,v. \end{cases}$$
		Observe that $\R'$ is always a strategy for Rei (i.e.\ $\R'(\r')\sub\supp(\r')$ for all $\r'$) since $\r_v>0$ and $\R$ is a strategy for Rei.
		
		First assume $p_u=1$, which means the optimal strategy $\R$ always has Rei choosing $u$ at $\r-\del_v$. By Fact~\ref{fact:recurrence} and our assumption $N^-(u)\neq\emptyset$, we have\begin{align}
			\S_D(\r-\del_v)= 1+\S_D(\r-\del_u-\del_v).\label{eq:switch1}
		\end{align} The possibly suboptimal strategy $\R'$ always has Rei choosing $v$ at $\r-\del_u$, so Fact~\ref{fact:recurrence} implies\begin{align}
			\S_D(\r-\del_u;\R')=\begin{cases} 1+\S_D(\r-\del_u-\del_v), &\text{if }N^-(v)\neq\emptyset\\ \S_D(\r-\del_u-\del_v), &\text{if }N^-(v)=\emptyset\end{cases}. \label{eq:switch2}
		\end{align} Hence, if $N^-(v)=\emptyset$, we have
		$$\S_D(\r-\del_u)\leq \S_D(\r-\del_u;\R') < \S_D(\r-\del_v)\leq \S_D(\r-\del_v)+\alpha(u,v),$$ 
		where the first inequality used that $\R'$ is possibly suboptimal, the second used \eqref{eq:switch1} and \eqref{eq:switch2}, and the last used that $\alpha(u,v)\geq 0$ always. If $N^-(v)\neq \emptyset$, we have $$\S_D(\r-\del_u)\leq \S_D(\r-\del_u;\R')= \S_D(\r-\del_v)< \S_D(\r-\del_v)+\alpha(u,v),$$ since in this case $\alpha(u,v)>0$. Either way contradicts our assumption, so we must have $p_u\neq 1$.
		
		Define $e_w=\sum_{w'\in N^+(w)} p_{w'}-\sum_{w'\in N^-(w)} p_{w'}$, which is the expected score when one round is played at $\r-\del_v$ when Rei follows $\R$ and Norman picks $w$. By Fact~\ref{fact:recurrence} and $\R$ being optimal, we have 
		\begin{equation}\S_D(\r-\del_v)=\max_{w}e_w+\sum_w p_w \S_D(\r-\del_v-\del_w).\label{eq:R}\end{equation} Similarly, we define $e'_w=\sum_{w'\in N^+(w)} p'_{w'}-\sum_{w'\in N^-(w)} p'_{w'}$.  By definition of $p'_w$ (which differs by either 0 or $\pm p_u$ from $p_w$), we have
		\begin{align*}e'_w-e_w&=\Big[\sum_{w'\in N^+(w)}p_{w'}'-p_{w'} \Big]- \Big[\sum_{w'\in N^-(w)}p_{w'}'-p_{w'} \Big]\\ &=\left[p_u \bone(v\in N^+(w))-p_u \bone(u\in N^+(w))\right]-\left[p_u \bone(v\in N^-(w))-p_u \bone(u\in N^-(w))\right]\\&\le p_u \bone(v\in N^+(w))+p_u \bone(u\in N^-(w))\\ &=p_u\bone(w\in N^-(v))+p_u\bone(w\in N^+(u))\le p_u \alpha(u,v),\end{align*}
		where this last step uses the definition of $\alpha(u,v)$. Using this together with Fact~\ref{fact:recurrence} gives
		\begin{equation}\S_D(\r-\del_u;\R')\leq \max_{w}e_w+ p_u\alpha(u,v)+\sum_{w\ne u,v} p_w \S_D(\r-\del_u-\del_w)+(p_u+p_v)\S_D(\r-\del_u-\del_v),\label{eq:R'}\end{equation}
		where implicitly we used that $\S_D(\r-\del_u-\del_w;\R')=\S_D(\r-\del_u-\del_w)$ since $\R'$ agrees with the optimal strategy $\R$ for every vector except $\r-\del_u$.  Since $\S_D(\r-\del_u)\leq \S_D(\r-\del_u;\R')$, we can combine \eqref{eq:R} and \eqref{eq:R'} and obtain
		\begin{align}
			\notag  \S_D(\r-\del_u)-\S_D(\r-\del_v)&\leq \S_D(\r-\del_u;\R')-\S_D(\r-\del_v)\\
			\notag &\leq p_u\alpha(u,v)+\sum_{w\ne u} p_w(\S_D(\r-\del_u-\del_w)-\S_D(\r-\del_v-\del_w))\\
			\label{eq:switch3}  &< p_u\alpha(u,v)+\sum_{w\ne u} p_w\alpha(u,v)\\
			\notag &=\alpha(u,v),
		\end{align} where the inequality \eqref{eq:switch3} used the assumptions that $p_u\neq 1$ (so $p_w>0$ for some $w\ne u$) and that $\r$ was a minimum counterexample to the desired strict inequality. As this contradicts our assumption, we conclude our desired strict inequality.
		
		This completes the proof in the case $N^-(u)\ne\emptyset$.  If $N^-(u)=\emptyset$ but $\alpha(u,v)=2$, the condition $p_u=1$ implies the following identity, which is slightly different from \eqref{eq:switch1},
		$$\S_D(\r-\del_v)= \S_D(\r-\del_v-\del_u).$$
		The identity \eqref{eq:switch2} still holds for $p_u=1$. Together they imply $|\S_D(\r-\del_u)-\S_D(\r-\del_v)|\leq 1<\alpha(u,v)$. So $p_u=1$ still contradicts the assumption for $\r$, and when $p_u\neq 1$ we can argue as above and conclude \eqref{eq:switch} with a strict inequality.  If we are in the situation where $N^-(u)=\emptyset$ and $\alpha(u,v)\neq 2$, then the same argument as above gives \eqref{eq:switch3} with a non-strict inequality, giving the desired result.
	\end{proof}
	
	We can use Lemma~\ref{lem:switch} to establish the optimal strategies for Rei for every digraph on at most 3 vertices.  Most of these cases are trivial.  Indeed, recall that a vertex $v$ in a digraph is called a \emph{source} if $N^-(v)=\emptyset$. If there is a source $v$ such that $N^+(v)$ contains all of the non-source vertices, then an optimal strategy for Norman is to always choose $v$ (which will allow him to win every round that Rei does not play a source, and Norman can only draw when Rei plays a source). Facing this strategy of Norman, every strategy of Rei gives the same outcome, hence is trivially optimal. One can check that such a source exists in every digraph with at most 3 vertices except for the circuit on 3 vertices and the directed path on 3 vertices, so it remains to address these two cases.
	
	Recall that if $D$ is the directed cycle $1\to 2\to 3\to 1$, then the $D$-game is just RPS. In  Section~1 we mentioned the following greedy strategy $\R_g$ for Rei in this semi-restricted $D$-game: When she still has all three options available, she chooses each with uniform probability; When she has two options, she chooses the stronger one with probability $2/3$ and the weaker one with probability $1/3$; When she has only one option, she chooses this with probability 1. The next theorem implies the first half of Theorem~\ref{thm:RPS}. 
	\begin{theorem}\label{thm:RPS2}
		The greedy strategy $\R_g$ is the unique optimal strategy for Rei in the semi-restricted $D$-game when $D$ is the directed cycle $1\to 2\to 3\to 1$.
	\end{theorem}
	\begin{proof}
		Let $\R$ be an optimal strategy for Rei. We shall prove that $\Pr(\R(\r)=u)=\Pr(\R_g(\r)=u)$ for all vertices $u$ and vectors $\r$ by induction on the quantity $\r_1+\r_2+\r_3$. The base case when $\r_1+\r_2+\r_3=0$ is trivial. Now suppose $\Pr(\R(\r')=u)=\Pr(\R_g(\r')=u)$ for all vertices $u$ and vectosr $\r'$ satisfying $\r'_1+\r'_2+\r'_3<\r_1+\r_2+\r_3$.  Let $p_u:=\Pr(\R(\r)=u)$ for each $u\in \{1,2,3\}$.
		
		When $|\supp(\r)|=1$, by the definition of a strategy for Rei, we have $\Pr(\R(\r)=u)=\Pr(\R_g(\r)=u)$ for all $u$.
		
		When $|\supp(\r)|=2$, we can assume without loss of generality that $\r_1,\r_2>0$.  Assume for contradiction $p_1=2/3+\epsilon$ with $\epsilon>0$, which means $p_2=1/3-\epsilon$. Applying Fact~\ref{fact:recurrence} to both $\R$ and $\R_g$ gives
		\begin{align*}
			&\S_D(\r;\R)=\left(\frac{1}{3}+2\epsilon\right) + \left(\frac{2}{3}+\epsilon\right) \S_D(\r-\del_1)+\left(\frac{1}{3}-\epsilon\right)\S_D(\r-\del_2),\\
			&\S_D(\r;\R_g)=\frac{1}{3}+\frac{2}{3}\S_D(\r-\del_1)+\frac{1}{3}\S_D(\r-\del_2),
		\end{align*} where here we used the assumptions that $\R$ is optimal and that $\R_g$ coincides with $\R$ on $\r'=\r-\del_u$ for $u=1,2$. These two identities imply
		$$\S_D(\r;\R_g)-\S_D(\r;\R)=\epsilon(\S_D(\r-\del_2)-\S_D(\r-\del_1) -2)<0,$$
		where the inequality is by Lemma~\ref{lem:switch} with $\alpha(2,1)=2$ (since $N^+(2)\cap N^-(1)=\{3\}$). This contradicts $\R$ being an optimal strategy, so we can not have $p_1=2/3+\epsilon$ for any $\epsilon>0$. If $p_1=2/3-\epsilon$ for some $\epsilon>0$, we now have
		\[\S_D(\r;\R)=\left(\frac{1}{3}+\epsilon\right) + \left(\frac{2}{3}+\epsilon\right) \S_D(\r-\del_1)+\left(\frac{1}{3}-\epsilon\right)\S_D(\r-\del_2).\]
		Similar to before we find
		$$ \S_D(\r;\R_g)-\S_D(\r;\R)=\epsilon(\S_D(\r-\del_1)-\S_D(\r-\del_2) -1)<0,$$
		where the inequality is again by Lemma~\ref{lem:switch}. This is also a contradiction, so we must have $p_1=2/3$ and $p_2=1/3$, which means $\Pr(\R(\r)=u)=\Pr(\R_g(\r)=u)$ for all $u$ as desired.
		
		When $|\supp(\r)|=3$, we let $\Delta:=\max\{p_2-p_3,p_3-p_1,p_1-p_2\}$. Note that $\Delta\geq 0$.  By Fact~\ref{fact:recurrence} and the assumption of $\R$ being optimal, we have
		$$\S_D(\r;\R)=\Delta+p_1\S_D(\r-\del_1)+p_2\S_D(\r-\del_2)+p_3\S_D(\r-\del_3).$$
		Our proof now divides into three cases.
		
		\medskip
		
		\noindent \emph{Case 1.} Exactly one term in $\{p_2-p_3,p_3-p_1,p_1-p_2\}$ equals $\Delta$.
		
		Without loss of generality, we can assume $p_2-p_3=\Delta$. Take
		$$\epsilon=\frac{1}{3}\min\{|\Delta-(p_3-p_1)|, |\Delta-(p_1-p_2)|\}>0.$$
		Consider a different strategy $\R'$ for Rei with
		
		$$\Pr(\R'(\r)=1)=p_1,\ \Pr(\R'(\r)=2)=p_2-\epsilon,\ \Pr(\R'(\r)=3)=p_3+\epsilon,$$
		and $\Pr(\R'(\r')=u)=\Pr(\R(\r')=u)$ for all vertices $u$ and vectors $\r'\ne \r$. By Fact~\ref{fact:recurrence} and our choice of $\epsilon$, we have
		$$\S_D(\r;\R')=(\Delta-2\epsilon)+p_1\S_D(\r-\del_1)+(p_2-\epsilon)\S_D(\r-\del_2)+(p_3+\epsilon)\S_D(\r-\del_3).$$
		Thus
		
		$$\S_D(\r;\R')-\S_D(\r;\R)=\epsilon(\S_D(\r-\del_3)-\S_D(\r-\del_2)-2)<0,$$
		where the inequality is by Lemma~\ref{lem:switch}. This contradicts $\R$ being optimal, so Case 1 can not happen.
		
		\medskip
		
		\noindent \emph{Case 2.}  Exactly two expressions in $\{p_2-p_3,p_3-p_1,p_1-p_2\}$ equal $\Delta$.
		
		Without loss of generality, we assume $p_2-p_3=p_3-p_1=\Delta$. This implies $p_3=1/3$, $p_2=1/3+\Delta$, and $p_1=1/3-\Delta$. Note that $\Delta>0$, as otherwise $p_1-p_2$ also attains $\Delta$. Similar to the above arguments, we can compute
		$$\S_D(\r;\R_g)-\S_D(\r;\R)=\Delta(\S_D(\r-\del_1)-\S_D(\r-\del_2)-1)<0,$$
		where the inequality is by Lemma~\ref{lem:switch}. This gives a contradiction, so Case 2 can not happen.
		
		\medskip
		
		\noindent \emph{Case 3.}  Every expression in $\{p_2-p_3,p_3-p_1,p_1-p_2\}$ equals $\Delta$.
		
		We have $p_2-p_3=p_3-p_1=p_1-p_2$, which implies $p_1=p_2=p_3=1/3$ and hence $\Pr(\R(\r)=u)=\Pr(\R_g(\r)=u)$ for all $u$. As neither of the previous two cases happen, we must be in this case, so we obtain the desired conclusion.
	\end{proof}
	
	The only remaining digraph on 3 vertices is the directed path, and in this case it turns out there are many optimal strategies for Rei.
	\begin{theorem}\label{thm:3path}
		The optimal strategies for Rei in the semi-restricted $D$-game when $D$ is the directed path $1\to 2\to 3$ are exactly those $\R$ satisfying $\Pr(\R(\r)=3)=1/2$ whenever $\{3\}\subsetneq \supp(\r)$.
	\end{theorem}
	That is, Rei will always play 3 with probability 1/2 provided she can play 3 and at least one other option. The intuition for Theorem~\ref{thm:3path} is as follows: in any particular round, suppose Rei plays $v$ with probability $p_v$.  It is clear that Norman will only ever play 1 or 2 for his turn, and for a given round these gives him expected payoffs of $p_2$ and $p_1-p_3$, respectively.  These two payoffs are equal precisely when $p_3=1/2$.  Thus the strategies in Theorem~\ref{thm:3path} are exactly those such that Norman is indifferent between which (reasonable) vertex he should play.  The proof of Theorem~\ref{thm:3path} is similar to that of Theorem~\ref{thm:RPS2}, and as such we have relegated its proof to Appendix~\ref{sec:path}.

	\section{Bounding the Score}\label{sec:analysis}
	For ease of notation, we will assume throughout this section that $k=|V(D)|$ is the number of vertices of our digraphs.
	
	\subsection{General bounds}
	Our proofs will need the following technical result, which will be used to control the expected number of rounds remaining in the semi-restricted $D$-game after one option has been depleted.
	\begin{lemma}\label{techineq}
		Let $p\in (0,1]$ and $N$ be an integer. Let $X=X_1+\dots+X_N$ where the $X_i$ are iid random geometric variables with parameter $p$. Then
		\[ \E[(\E X-X)\cdot \bone(X\leq \E X)]=O(\sqrt{N}/p)\]
	\end{lemma}
	\begin{proof}
		Note that $\E[X]=N/p$.  Janson's lower tail bound~\cite[Theorem 3.1]{janson2018tail} states that for $\lambda\le 1$ we have
		\[ \Pr(X\le \lambda N/p)\le \exp(-N(\lambda-1-\ln \lambda)).\]
		Therefore,
		\begin{align*}
			\E[(N/p-X)\bone(X\leq N/p)]= \sum_{i\le N/p} \Pr(X\le i)&=(N/p+o(N))\int_0^1  \Pr(X\le \lambda N/p)d\lambda \\&\le (N/p+o(N))\int_0^1 \exp(-N(\lambda-1-\ln \lambda))d\lambda.
		\end{align*}
		So it suffices to show that the integral on the right hand side is $O(1/\sqrt{N})$, which follows from a routine asymptotic estimation.
\end{proof}

We can use this technical result to prove Theorem~\ref{thm:UpperUniform}.
\begin{proof}[Proof of Theorem~\ref{thm:UpperUniform}]
	Recall that we wish to show
	\[\S_D(n\cdot \1)= \max_v\{d^+(v)-d^-(v)\}n+O_D(\sqrt{n}).\]
	The lower bound $\S_D(n\cdot \1)\ge  \max_v\{d^+(v)-d^-(v)\}n$ follows form Norman deterministically choosing $v$ every round.  It remains to exhibit a strategy of Rei such that Norman can gain at most $\max_v\{d^+(v)-d^-(v)\}n+O_D(\sqrt{n})$ points in expectation against this strategy.
	
	Intuitively, Rei's strategy will be to play uniformly at random until some option is depleted, after which she plays arbitrarily.  To more formally analyze this strategy, we will have Rei generate an infinite random string $\pi=\pi_1\pi_2\cdots$ where each $\pi_t$ is chosen uniformly and independently amongst $V(D)$.  In the $t$-th round, Rei will play $\pi_t$, unless that option has been depleted, in which case she plays arbitrarily.
	
	For $w\in V(D)$, let $T_w$ denote the smallest integer $t$ such that $|\{s\le t:\pi_s=w\}|=n$, and let $T=\min_w T_w$.  Note that before the $T$-th round of the game, Rei can still play every option (and hence does so uniformly), and that some option is depleted after completion of the $T$-th round. Thus for $t\le T$, if Norman plays $v$ in the $t$-th round, then the expected increase in score for Norman is exactly $k^{-1}(d^+(v)-d^-(v))$, so his maximum expected increase in score is at most $k^{-1}\max_v\{d^+(v)-d^-(v)\}$. Using this and that Norman's score increases by at most 1 for each round after the $T$-th, we find
	\begin{align}
		\E[\S_D(\r)|T]&\le T\cdot k^{-1}\max_v\{d^+(v)-d^-(v)\}+(kn-T)\nonumber\\
		&\le \max_v\{d^+(v)-d^-(v)\}n+(kn-T),\label{eq:negbin}
	\end{align}
	where this last step used that deterministically $T\le kn$.  It remains to upper bound the expected value of $kn-T$.
	
	Observe that $T_w$ has the distribution of the sum of $n$ independent geometric random variables with parameter $1/k$.  Thus $\E T_w=kn$, and Lemma \ref{techineq} gives
	\[\E[(kn-T_w)\bone(T_w\leq kn)]=O_D(\sqrt{n}).\]
	It follows that
	\[ \E[kn-T]= \E[ \max_w \{kn-T_w\}]\le \sum_{w} \E[(kn-T_w)\bone(T_w\leq kn)]=O_D(\sqrt{n}).\]
	Combining this with \eqref{eq:negbin} and using the tower property of conditional expectation gives the desired result.
\end{proof}

A similar approach gives Theorem~\ref{thm:UpperNonuniform}.
\begin{proof}[Proof of Theorem~\ref{thm:UpperNonuniform}]
	Recall that we wish to prove \[\S_D(\r)\le\max_v\left\{\sum_{u\in N^+(v)}\r_u-\sum_{u\in N^-(v)}\r_u\right\}+O_D(M^{2/3}),\]
	where $M=\max_v \r_v$. The approach we use is very similar to that of Theorem~\ref{thm:UpperUniform}, so we will omit some of the redundant details.  Intuitively, we will prove the lower bound by having Rei play each option $w$ with probability $\r_w/\sum_u \r_u$ until an option is depleted, after which she plays arbitrarily.  However, we will first need to ``trim'' $\r$ to ignore vertices with $\r_v$ small.
	
	To be more precise, let $\r'$ be the vector defined by $\r'_w=\r_w$ if $\r_w\ge M^{2/3}$ and $\r'_w=0$ otherwise.  Rei generates an infinite random string $\pi=\pi_1\pi_2\cdots$ where $\pi_t=w$ with probability $\r'_w/\sum_u \r'_u$ for all $w\in \supp(\r')$ independently of every other $\pi_s$.  In the $t$-th round Rei will play $\pi_t$ unless that option is depleted, in which case she plays arbitrarily.   
	
	For $w\in \supp(\r')$, let $T_w$ denote the smallest integer $t$ such that $|\{s\le t:\pi_s=w\}|=\r'_w$, and let $T=\min_{w\in \supp(\r')} T_w$.  Note that for $t\le T$, if Norman plays $v$ in the $t$-th round then the expected increase in score for Norman is exactly $(\sum_u \r_u')^{-1}(\sum_{u\in N^+(v)}\r'_u-\sum_{u\in N^-(v)}\r'_u)$.  Using this, that Norman's score increases by at most 1 each round, and that $T\le \sum_u \r'_u$, we see that
	\begin{equation}\E[\S_D(\r)|T]\le \max_{v}\left\{\sum_{u\in N^+(v)}\r'_u-\sum_{u\in N^-(v)}\r'_u\right\}+\lrpar{\sum_u \r_u-T}.\label{eq:negbin2}\end{equation}
	
	Observe that each $T_w$ random variable has the distribution of the sum of $\r_w'=\r_w$ independent geometric random variables with parameter $\r_w/\sum_u \r_u'$.  Thus $\E T_w=\sum_u \r_u'$, and Lemma~\ref{techineq} gives
	\[\E[(\sum_u \r_u'-T_w)\bone(T_w\leq \sum_u \r_u')]=O_D(\r_w^{-1/2}\sum_u \r_u')=O_D(M^{2/3}),\]
	where this last step used $\r_w\ge M^{2/3}$ and $\sum_u \r_u'\le k\cdot M$.
	It follows that
	\[ \E[\sum_u \r_u'-T]= \E[ \max_w\{ \sum_u \r_u'-T_w\}]\le \sum_w \E[(\sum_u \r_u'-T_w)\bone(T_w\leq \sum_u \r_u')]=O_D(M^{2/3}).\]
	Combining this with \eqref{eq:negbin2}, the tower property of conditional expectation, and that $\sum_u \r_u-\sum_u \r_u'=O_D(M^{2/3})$ gives the desired result.
\end{proof}

\subsection{Spectral bounds}
As we saw in the case of RPS, there are semi-restricted games where Rei's optimal strategy is to start by playing uniformly at random until some option runs out. In this section, we will classify a class of graphs for which Rei's optimal strategy is to  play approximately uniformly until some option runs out.
We do so by identifying a property which implies that the expected loss at any given step is proportional to its deviation from the uniform strategy. 

Given a digraph $D$, let $A_D$ be the \textit{skew adjacency matrix} of $D$, i.e.
\[ (A_D)_{ij}= \begin{cases}
	1, &\text{if } ij\in E(D) \\
	-1 &\text{if } ji\in E(D)\\
	0 & \text{otherwise}
\end{cases}
\]
Our main focus for this subsection will be digraphs satisfying
\[ \nul(A_D)=\rmspan(\1).\]
To motivate this, observe that at any given state $\r$, if $p$ is the probability vector such that Rei plays vertex $u$ with probability $p_u$, then Rei's expected loss for this round will be 0 if and only if $A_Dp=0$. Therefore, for digraphs with $\nul(A_D)=\rmspan(\1)$, the steps where Rei's expected loss is 0 are the ones where Rei picks her option uniformly at random. We will now show that in fact the expected loss for such digraphs is always proportional to the maximum deviation between $p_v$ and $1/k$.

\begin{lemma}\label{punish}
	For any $k$-vertex digraph $D$ with $\nul(A_D)=\rmspan(\1)$, there exists a constant $\alpha_D>0$ such that for any probability vector $p$ indexed by $V(D)$, we have
	\[\max_v \left\{ \sum_{u\in N^+(v)}p_u-\sum_{u\in N^-(v)}p_u\right\}\geq\alpha_D\max_v |p_v-1/k|.\]
\end{lemma}
\begin{proof}
	First notice that the left-hand side of our inequality satisfies
	\begin{equation} \max_v \left\{ \sum_{u\in N^+(v)}p_u-\sum_{u\in N^-(v)}p_u\right\}=\max_v (A_Dp)_v\ge k^{-1}\max_v |(A_Dp)_v|,\label{eq:punish}\end{equation}
	where the last inequality holds because $\sum_v (A_Dp)_v=0$
	and because $x = (x_1,...,x_k)$ with $\sum_i^k x_i = 0$ and $x_1\geq...\geq x_k$ satisfies
	\[kx_1\geq \sum_{i:\ x_i\geq 0} x_i = \sum_{i:\ x_i\leq 0} |x_i| \geq |x_k|.\]
	Since $A_D$ is skew-symmetric and $\dim(\nul(A_D)) = 1$, it admits an eigendecomposition $A_D = \sum_{i=1}^k\lambda_i e_ie_i^*$ with $0=|\lambda_1|<|\lambda_2|\leq...\leq|\lambda_k|$, $e_i^*e_i=1$, and $e_i^* e_j=0$ for all $i\ne j$. Notice that $A_Dp = A_D(p -k^{-1}\cdot\1)$, we have 
	\begin{align*}
		\norm{A_Dp} &= \norm{A_D\lrpar{p-k^{-1}\cdot\1}} = \norm{\sum_{i=2}^k e_i\left[\lambda_i e_i^*\lrpar{p-k^{-1}\cdot\1 }\right]} \\&= \sqrt{\sum_{i=2}^k \lrabs{\lambda_i}^2 \lrabs{e_i^*\lrpar{p-k^{-1}\cdot\1 }}^2}\geq \lrabs{\lambda_2}\cdot\norm{p-k^{-1}\cdot\1} \ge \lrabs{\lambda_2} \max_v \lrabs{p_v-1/k}.
	\end{align*}
	Hence 
	\[\max_v|(A_Dp)_v| \geq \frac{1}{k}\norm{A_Dp} \geq \frac{|\lambda_2|}{k}\max_v \lrabs{p_v-1/k}, \] 
	so by \eqref{eq:punish} the statement follows with $\alpha_D=\lrabs{\lambda_2}/k^2$.
\end{proof}
Using this we can prove the following lower bound on $\S_D(n\cdot \1)$ whenever $\nul(A_D)=\rmspan(\1)$.

\begin{theorem}\label{lowergood}
	If $D$ is a digraph on at least 2 vertices with $\nul(A_D)=\rmspan(\1)$, then \[\S_D(n\cdot \1)=\Omega(\sqrt{n}).\]
\end{theorem}

\begin{proof}
	We identify $V(D)$ with $[k]$ for $k>2$ and let $N=kn$ be the total number of rounds in this game. Fix some optimal strategy $\R$ for Rei. We divide the game into two phases: the first phase is when every option is still available to Rei, and the second phase is when at least one option has been depleted. Let $T$ be the (random) number of steps in the first phase.
	
	Let $\X=(X_1, X_2,\dots, X_N)$ be the stochastic process where $X_i$ denotes Rei's choice when $i\le T$, and for $i>T$ the $X_i$ are independently uniformly random over $[k]$. We define \[p^t_v=\Pr(X_t=v| X_1, X_2,\dots, X_{t-1}).\]
	For $t\le T$, $p^t_v$ corresponds to the probability of Rei choosing $v$ at the $t$-th round given her previous choices.
	
	We consider the following greedy strategy for Norman: Suppose at the $t$-th round, the restriction vector for Rei is $\r$ with $p_v=\Pr(\c{R}(\r)=v)$ for all $v$, then Norman will deterministically pick $w$ that maximizes his expected gain, i.e. 
	\[ w=\arg \max_v\left\{ \sum_{u\in N^+(v)}p_u-\sum_{u\in N^-(v)}p_u\right\}.\] Here we break ties arbitrarily. If $t\leq T$, we have $p_v=p^t_v$. Then, by Lemma \ref{punish}, Norman's expected gain at this round is
	\[ \max_v\left\{ \sum_{u\in N^+(v)}p_u-\sum_{u\in N^-(v)}p_u\right\}=\max_v \left\{ \sum_{u\in N^+(v)}p^t_u-\sum_{u\in N^-(v)}p^t_u\right\}\geq \alpha_D|p^t_1-1/k|.\] If $t>T$, there exists some $v$ with $p_v=0$. Then, by Lemma \ref{punish}, Norman's expected gain at this round is
	\[ \max_v\left\{ \sum_{u\in N^+(v)}p_u-\sum_{u\in N^-(v)}p_u\right\}\geq\alpha_D/k \]
	
	We let $S_1$ denote Norman's expected score during the first phase under this strategy and $S_2$ his score during the second phase, so that his total expected score is $S_1+S_2$. By above analysis, we have
	\begin{align}
		S_1&\geq \E\left[\sum_{t=1}^T \alpha_D|p^t_1-1/k|\right]= \alpha_D \E\left[\sum_{t=1}^N |p^t_1-1/k|\right], \label{greedyineq}\\
		S_2&\geq \E\left[\sum_{t=T+1}^N \alpha_D/k\right]=\frac{\alpha_D}{k}\cdot\E[N-T].\label{greedyineq2}
	\end{align}
	
	From \eqref{greedyineq} and \eqref{greedyineq2}, we see that $S_1,S_2\ge 0$, and thus to prove the result it suffices to show $\max\{\E \sum_t |p_1^t-1/k|,\E[N-T]\}=\Omega(\sqrt{n})$. To this end, fix some small constant $c>0$ to be chosen later. We may assume from now on that
	\begin{equation}\E \left[\sum_{t=1}^N |p_1^t-1/k|\right]\le c \sqrt{n},\label{assumption}\end{equation}
	as otherwise we are done by \eqref{greedyineq}.  
	
	We claim that, if $c$ is sufficiently small, then with probability at least 1/8 the vertex 1 appears at least $n+\Omega(\sqrt{n})$ times in $\X$.  This will imply our theorem.  Indeed, if 1 appears $n+Y$ times in $\X$, then necessarily $N-T\ge Y$ since $(X_1,\ldots,X_{T-1})$ by definition has each symbol appearing less than $n$ times.  Thus this claim together with \eqref{greedyineq2} implies $S_2=\Omega(\sqrt{n})$, giving the result.  It remains to prove the claim.
	
	Let $f(\X)$ be the number of times $1$ occur in $\X$. We will show that we can couple $f(\X)$ with a binomial random variable $\Bin(N,1/k)$ such that 
	\[ \E [|f(\X)-\Bin(N,1/k)|]\le \E \left[\sum_{t=1}^N|p^t_1-1/k|\right].\]
	To do so, we consider independent uniform random variables $U_1,U_2,\dots, U_N\sim U[0,1]$. Let $Y_t= \bone(U_t\le 1/k)$. Notice that 
	\[ \Bin(N,1/k)\sim \sum_{1\le t\le N} Y_t. \]
	To couple $X_t$ with $Y_t$, we let 
	\[ X_{t}=j \hspace{.6em}\text{ if }\hspace{.6em} U_{t}\in \left(\sum_{i<j} p_i^t, \sum_{i\le j} p_i^t\right].\]
	With this coupling,
	\[ \Pr(\bone(X_t=1)\neq Y_t) =\Pr(U_t\le p_1^t \text{ and } U_t> 1/k)+\Pr(U_t>p_1^t \text{ and } U_t\le 1/k) = \E[ |p^t_1-1/k|].\]
	Therefore 
	\[ \E[|f(\X)-\Bin(N,1/k)|]= \E\left[ \bigg|\sum_{t=1}^N \bone(X_t=1)- Y_t\bigg|\right]\le \sum_{t=1}^N \E [|p^t_1-1/k|]\le c \sqrt{n},\]
	where this last step used \eqref{assumption}.
	We will now exploit the deviation of $\Bin(N,1/k)$ from its mean $N/k=n$ to show that $f(\X)$ also deviates from $n$. Since $k\ge 2$, we can pick a small enough $\beta>0$ such that
	\[ \Pr(\Bin(N,1/k)\ge n+\beta \sqrt{n})\ge \frac14.\]
	Taking $c=\beta/16$ together with Markov's inequality gives
	\[ \Pr(\Bin(N,1/k)-f(\X)\ge \beta\sqrt{n}/2)\le \frac{\E[ |f(\X)-\Bin(N,1/k)|]}{\beta\sqrt{n}/2} \le 2c/\beta= \frac18. \]
	It follows that 
	\begin{align*}
		\Pr(f(\X)\ge n+\beta\sqrt{n}/2)&\ge \Pr(\Bin(N,1/k)\ge n+\beta \sqrt{n} \text{ and } \Bin(N,1/k)-f(\X)<\beta\sqrt{n}/2)\\
		&\ge \Pr(\Bin(N,1/k)\ge n+\beta\sqrt{n}) -\Pr(\Bin(N,1/k)-f(\X)\ge \beta\sqrt{n}/2)\\
		&\ge 1/8. 
	\end{align*}
	This concludes the proof of the claim, hence also the proof of this theorem.
\end{proof}
To apply Theorem~\ref{lowergood}, we need to find a class of digraphs satisfying the necessary spectral conditions. Eulerian tournaments turn out to be such a class, and to establish this we need the following lemma.

\begin{lemma}\label{eventournament}
	Every tournament $T$ on an even number of vertices has $\det(A_T)\neq 0$.
\end{lemma}
\begin{proof}
	Suppose $T$ is a tournament on $n$ vertices. We will prove that $\det(A_T)$ is an odd number, and hence is not 0. Let $I$ be the $n\times n$ identity matrix and $J$ be the $n\times n$ all $1$ matrix. Because $J$ has eigenvalue $n$ with multiplicity 1 and 0 with multiplicity $n-1$, we have
	\[ \det(A_T)\equiv \det(J-I)\equiv (-1)^{n-1}(n-1)\equiv 1\mod 2,\]
	where this last step holds when $n$ is even.
\end{proof}

\begin{lemma}\label{oddtournament}
	All Eulerian tournaments $D$ satisfy $\nul(A_D)=\rmspan(\1)$.
\end{lemma}
\begin{proof}
	First notice that Eulerian tournaments have an odd number of vertices. Clearly $\1\in \nul(A_D)$.  Let $T$ be the tournament obtained by deleting an arbitrary vertex of $D$.  By Lemma \ref{eventournament}, \[\rk(A_D)\ge \rk(A_T)=|V(D)|-1,\] which means $\dim(\nul(A_D))=1$.
\end{proof}

\begin{proof}[Proof of Theorem \ref{eulerian}]
	The upper bound follows from Theorem~\ref{thm:UpperUniform}
	while the lower bound follows from Theorem \ref{lowergood} and Lemma \ref{oddtournament}.
\end{proof}

\section{Oblivious Strategies}\label{sec:oblivious}
Inspired by the optimal strategy for semi-restricted RPS, We say that a strategy for Rei $\c{R}$ is \textit{oblivious} if the random variable $\c{R}(\r)$ depends only on $\supp(\r)$ for all $\r$, i.e.\ if Rei plays based only on the set of options she can play without taking into account how many times she can perform each option.  We say that a digraph $D$ is \textit{oblivious} if Rei has an optimal strategy in the semi-restricted $D$-game which is oblivious.  We emphasize that $D$ being oblivious only guarantees that there exists at least one oblivious optimal strategy, not that every optimal strategy is oblivious.

Naively, oblivious strategies seem like they would be ineffective.  However, in Section \ref{sec:optimal} we saw that every digraph on at most 3 vertices is oblivious, and the unique optimal strategy in the game considered in~\cite{spiro2022online} was also oblivious.  Thus at this point one might guess that every digraph is oblivious.

Unfortunately this turns out not to be the case.  In fact, we will show that almost every tournament and almost every Eulerian tournament fails to be oblivious.  We do this through the following proposition.
\begin{proposition}\label{prop:oblivious}
	Let $D$ be a tournament.  If there exists a set $S\sub V(D)$ with $|S|$ even and with $N^+(v)\cap S\not\subseteq N^+(w)\cap S$ for all $v\in S,\ w\in V(D)$, then $D$ is not oblivious.
\end{proposition}

To prove this, we need the following technical result.
\begin{lemma}\label{oblistruct}
	Let $D$ be an oblivious digraph and $\c{R}$ an oblivious optimal strategy for Rei in the semi-restricted $D$-game.  Let $S\sub V(D)$,  and let $p$ be the probability vector satisfying $p_u=\Pr(\c{R}(\r)=u)$ whenever $\supp(\r)=S$.   If $v\in S$ is such that there exists a non-negative vector $q$ with $\supp(q)=S$ and $\sum_{u\in N^+(v)} q_u-\sum_{u\in N^-(v)} q_u=\max_w \{\sum_{u\in N^+(w)} q_u-\sum_{u\in N^-(w)} q_u\}$, then we have
	\[\sum_{u\in N^+(v)} p_u-\sum_{u\in N^-(v)} p_u=\max_w \left\{\sum_{u\in N^+(w)} p_u-\sum_{u\in N^-(w)} p_u\right\}.\]
\end{lemma}
Roughly speaking, this proposition says that if $v$ can achieve $\max_w \{\sum_{u\in N^+(w)} p_u-\sum_{u\in N^-(w)} p_u\}$, then it will achieve this maximum.  This result agrees with Theorems~\ref{thm:RPS} and \ref{thm:3path} where the optimal oblivious strategies have as many vertices $v$ achieving this maximum as possible.

\begin{proof}
	Assume $v,q$ are as in the hypothesis of the proposition, and let $n$ be a large integer.  Define the vector $\r$ by having\footnote{Strictly speaking we should define $\r_u=\lceil q_un\rceil$ to make $\r$ a restriction vector, but this distinction will not affect our analysis.} $\r_u=q_u n$ for all $u$, noting that $\supp(\r)=\supp(q)=S$.  Let $\gamma=\min_u q_u/p_u$, where we let $q_u/p_u:=\infty$ if $p_u=0$.  Note that $0<\gamma<\infty$ since $\supp(q)=S$ and $p\ne 0$.
	\begin{claim}
		With probability tending towards 1 as $n$ tends towards infinity, an option will be depleted after $(\gamma+o(1)) n$ rounds, at which point every vertex $u$ has $(q_u-p_u \gamma)n+o(n)$ actions remaining in expectation. 
	\end{claim}
	\begin{proof}[Proof of claim]
		The analysis is similar to the one performed in the proof of Theorem \ref{thm:UpperNonuniform}.
		
		Rei generates an infinite random string $\pi=\pi_1\pi_2\dots$ where $\pi_t=u$ independently with probability $p_u$. In the $t$-th round Rei will play $\pi_t$ until an option is depleted.
		
		Let $T_u$ denote the smallest integer $t$ such that $|\{s\le t: \pi_s=u\}|=\r_u$, and let $T=\min_{u\in \supp(\r)} T_u$. Note that $T$ is the number of rounds until an option is depleted. 
		
		For any vertex $u$, we have  $T_u\sim X_1+\ldots+X_{r_u}$ where $X_i\sim \mathrm{Geo}(p_u)$ iid. Therefore $\E T_u=r_u/p_u$ and $\Var(T_u)=o((\E T_u)^2)$. It follows from Chebyshev's inequality that, for example,  
		\[ \Pr(|T_u-\E T_u|\ge (\E T_u)^{2/3})=o(1). \]
		By a union bound, we see that with high probability $T_u= (1+o(1)) \E T_u$ for all $u$, and hence $T=(\gamma+o(1)) n$ with high probability.
		
		Moreover, by a routine application of the second moment method and union bound, with high probability the number of times $u$ appears in the first $(\gamma+o(1)) n$ letters in $\pi$ is $(\gamma+o(1)) p_u n$ for all $u$, so after the first vertex is depleted, each vertex $u$ has $(q_u-p_u \gamma)n+o(n)$ actions remaining in expectation.
	\end{proof}
	
	Consider the following (possibly non-optimal) strategy for Norman: pick an arbitrary vertex $w$ and play this until some option is depleted, then play $v$ for the rest of the game.  By the claim above, the expected score for Norman under this strategy is
	\[ \left(\sum_{u\in N^+(w)}p_u -\sum_{u\in N^-(w)}p_u \right)\gamma n+\left(\sum_{u\in N^+(v)} (q_u-p_u\gamma)n-\sum_{u\in N^-(v)} (q_u-p_u \gamma)n\right)+o(n).\]
	The quantity above is a lower bound for $\S_D(\r)$.  By using Theorem~\ref{thm:UpperNonuniform} and the hypothesis on $q$, we have
	\[\S_D(\r)\le  \left(\sum_{u\in N^+(v)} q_u-\sum_{u\in N^-(v)} q_u\right)n+O(n^{2/3}).\]
	Comparing these two inequalities for $\S_D(\r)$ gives
	\[\left(\sum_{u\in N^+(w)}p_u- \sum_{u\in N^-(w)}p_u \right)\gamma n\leq \left(\sum_{u\in N^+(v)} p_u-\sum_{u\in N^-(v)} p_u \right)\gamma n+o(n).\]
	This is only possible if $\sum_{u\in N^+(w)} p_u-\sum_{u\in N^-(w)} p_u\leq \sum_{u\in N^+(v)} p_u-\sum_{u\in N^-(v)} p_u$.  As $w$ was arbitrary, this implies the result.
\end{proof}

We can now prove Proposition~\ref{prop:oblivious}.

\begin{proof}[Proof of Proposition~\ref{prop:oblivious}]
	Assume for contradiction that an oblivious optimal strategy $\c{R}$ existed, and let $p$ be the probability vector satisfying $p_u=\Pr(\c{R}(\r)=u)$ whenever $\supp(\r)=S$.  
	
	For each $v\in S$, consider the non-negative vector $q$ with $q_u=2|V(D)|$ for $u\in N^+(v)\cap S$, $q_u=1$  for $u\in S\sm N^+(v)$, and $q_u=0$ otherwise.  By hypothesis, for all $w\ne v$ we have \[|N^+(w)\cap (N^+(v)\cap S)|\le |N^+(v)\cap S|-1.\] This implies for all $w\ne v$ that
	\begin{align*}\sum_{u\in N^+(w)} q_u-\sum_{u\in N^-(w)} q_u&\le \sum_{u\in N^+(w)} q_u\le |N^+(w)\cap (N^+(v)\cap S)|\cdot 2|V(D)|+|V(D)|\cdot 1\\&\le |N^+(v)\cap S|\cdot 2|V(D)|-|V(D)|\le \sum_{u\in N^+(v)} q_u-\sum_{u\in N^-(v)} q_u.\end{align*}
	Because $\supp(q)=S$, Lemma~\ref{oblistruct} implies
	\begin{equation}\sum_{u\in N^+(v)} p_u-\sum_{u\in N^-(v)} p_u=\max_w \left\{\sum_{u\in N^+(w)} p_u-\sum_{u\in N^-(w)} p_u\right\}\label{eq:vMax}\end{equation}
	for all $v\in S$.
	
	Since $\supp(p)\sub S$, we have
	\[\sum_{v\in S}\left(\sum_{u\in N^+(v)} p_u-\sum_{u\in N^-(v)} p_u\right)=\sum_{v\in S}\left(\sum_{u\in N^+(v)\cap S} p_u-\sum_{u\in N^-(v)\cap S} p_u\right)=0.\]
	This together with the fact that \eqref{eq:vMax} holds for all $v\in S$ implies that we must have
	\begin{equation}\label{nullsubeven}
		\sum_{u\in N^+(v)} p_u-\sum_{u\in N^-(v)} p_u=\sum_{u\in N^+(v)\cap S} p_u-\sum_{u\in N^-(v)\cap S} p_u=0
	\end{equation}
	for all $v\in S$. Define $A_D[S]$ to be the submatrix of $A_D$ restricted to the rows and columns indexed by $S$. Then, \eqref{nullsubeven} states that $p$ restricted to the indices of $S$ is a nullvector for $A_D[S]$, and $p$ being a probability vector means that it is a non-zero nullvector. But from Lemma \ref{eventournament} we know that such a vector $p$ does not exist, giving the desired contradiction.
\end{proof}

Using Proposition~\ref{prop:oblivious}, we can quickly establish that an exponentially small proportion of tournaments are oblivious.

\begin{theorem}\label{thm:randTournament}
	Let $\c{T}_n$ denote the set of tournaments on $[n]$.  There exists a constant $c>0$ such that if $n$ is sufficiently large and $D$ is chosen uniformly at random from $\c{T}_n$, then
	\[\Pr(D\textrm{ is oblivious})\le e^{-cn}.\]
\end{theorem}

\begin{proof}
	Let $m=n$ if $n$ is even, and $m=n-1$ otherwise. Fix any subset $S\sub [n]$ of size $m$.  Let $B$ denote the event that there exists $v\in S,\ w\in V(D)$ such that $N^+(v)\cap S\subseteq N^+(w)\cap S$.  By Proposition~\ref{prop:oblivious}, if $D$ is oblivious, then $B$ must occur.  Thus it suffices to upper bound $\Pr(B)$.
	
	For any $v,w\in V(D)$, it is straightforward to show that \[\Pr(N^+(v)\cap S\subseteq N^+(w)\cap S)\le (3/4)^{n-3}.\]
	Thus by a union bound, we find
	\[\Pr(D\textrm{ is oblivious})\le \Pr(B)\le n^2\cdot (3/4)^{n-3},\]
	giving the desired result.
\end{proof}

Similarly, almost every Eulerian tournament is not oblivious.

\begin{theorem}\label{thm:randomEulerian}
	Let $\c{E}_n$ denote the set of Eulerian tournaments on $[2n+1]$.  If $D$ is chosen uniformly at random from $\c{E}_n$, then
	\[\Pr(D\textrm{ is oblivious})=O(n^{-2}).\]
\end{theorem}
Proving Theorem~\ref{thm:randomEulerian} requires a significantly more complicated argument than that of Theorem~\ref{thm:randTournament}.  Philosophically, this is because it is easy to sample a uniformly random tournament (one can just choose the orientation of each arc uniformly and independently), but it is much harder to generate an Eulerian tournament uniformly at random, see for example \cite{mckay1998asymptotic} which implicitly provides an efficient probabilistic sampling algorithm using random walks.  We get around this issue by invoking a ``switching'' type argument, the full details of which can be found in Appendix~\ref{sec:randomEulerian}.

\section{Open Problems}\label{sec:conclusion}
Many open questions about semi-restricted games remain.  We discuss a few of these in the following subsections.

\subsection{Optimal Strategies}
In this paper, we determined all of the optimal strategies for semi-restricted $D$-games when $D$ has at most 3 vertices.  It would be of interest to do this for some (non-trivial) infinite family of digraphs as well.  Perhaps the simplest such family is the following.
\begin{problem}
	Determine all of the optimal strategies for the semi-restricted $P_n$-game, where $P_n$ denotes the directed path on $n$ vertices.
\end{problem}

In Theorem~\ref{thm:randomEulerian}, we showed that almost every Eulerian tournament $D$ fails to have a ``simple'' (i.e.\ oblivious) optimal strategy for Rei in the semi-restricted $D$-game.  A closer inspection of the proof shows that these tournaments fail to have such an optimal strategy when Rei has $|V(D)|-1$ remaining options.  Thus it is possible (though seemingly unlikely) that Rei has a simple optimal strategy when every option is still available to her.

\begin{problem}
	Is it true that for every Eulerian tournament $D$, there exists an optimal strategy for Rei in the semi-restricted $D$-game such that under this strategy, if Rei has restriction vector $\r$ with $\supp(\r)=V(D)$, then Rei plays each option with probability $1/|V(D)|$?
\end{problem}

\subsection{Improved Bounds}
Theorem~\ref{thm:RPS} shows that the expected score for Norman in semi-restricted RPS is $\Theta(\sqrt{n})$ when both players play optimally.  It would be interesting to get a more precise estimate.
\begin{question}\label{quest:RPS}
	Does there exist a $c>0$ such that $\S_D(n,n,n)\sim c \sqrt{n}$ when $D$ is the directed 3-cycle?  If so, what is $c$?
\end{question}
Because Theorem~\ref{thm:RPS} gives the optimal strategy for Rei in this game, Question~\ref{quest:RPS} is equivalent to asymptotically determining the expected value of some strange (but explicit) random variable. We have computed the exact value of $\S_D(n,n,n)$ for $n\le 100$, and this data suggests that we may have $c\approx 1.46$.

We proved effective general upper bounds on $\S_D(\r)$ through Theorems~\ref{thm:UpperUniform} and \ref{thm:UpperNonuniform}.  It is natural to ask if these results can be improved.  For example, it might be possible to improve the $O_D(M^{2/3})$ error term in Theorem~\ref{thm:UpperNonuniform} to $O_D(M^{1/2})$.

\begin{question}
	Is it true that for every digraph $D$ and restriction vector $\r$ with $M=\max_v \r_v$, we have
	\[ \S_D(\r)\le \max_v\left\{\sum_{u\in N^+(v)}\r_u-\sum_{u\in N^-(v)}\r_u\right\}+O_D(M^{1/2}).\]
\end{question}
Such a result would be an (optimal) strengthening of Theorem~\ref{thm:UpperNonuniform}, as well as a significant generalization of Theorem~\ref{thm:UpperUniform}.  The central obstacle with this question is that we do not know what strategy Rei should use so that Norman can obtain at most this many points in expectation.

When $D$ is an Eulerian digraph, Theorem~\ref{thm:UpperUniform} shows that $0\le S_D(n\cdot \1)\le O_D(\sqrt{n})$.  We think that this upper bound might be tight in general.

\begin{question}\label{quest:lowerBound}
	Does every Eulerian digraph with at least one arc satisfy $\S_D(n\cdot \1)=\Omega_D(\sqrt{n})$?
\end{question}
A positive answer to this question together with Theorem~\ref{thm:UpperUniform} would show $\S_D(n\cdot \1)=\Theta_D(\sqrt{n})$ for all such $D$.

We can show Question~\ref{quest:lowerBound} has a positive answer whenever the spectral conditions of Theorem~\ref{lowergood} are satisfied. Motivated by this, we ask when exactly these spectral conditions occur.  We emphasize that this is a purely linear algebraic question and does not involve any game theory.
\begin{question}\label{quest:spectral}
	Which Eulerian digraphs $D$ are such that their skew-adjacency matrix $A_D$ has a nullspace of dimension 1?
\end{question}
For example, Lemma~\ref{oddtournament} shows that Eulerian tournaments have this property, and we can also show that this holds for certain powers of directed Hamiltonian cycles.  Similarly it is not difficult to show that $D$ will fail to satisfy this property if its underlying graph is bipartite, has an even number of vertices, or if there exist distinct vertices $v,v'$ with $N^+(v)=N^+(v')$ and $N^-(v)=N^-(v')$.

We note that the matrix $A_D$, as well as the essentially equivalent Hermitian matrix $i\cdot A_D$, have been well studied in the literature, see for example \cite{guo2017hermitian,liu2015hermitian}, and it is possible that the results and techniques used in these papers could give insight into the answer to Question~\ref{quest:spectral}.

\textbf{Acknowledgments.}  The authors are indebted to Yuanfan Wang for many fruitful discussions on this topic.  We would also like to  thank Jack J.\  Garzella, as well as the Mathematics Graduate Student Council at UCSD, for organizing the workshop where this work was initiated.  We thank Krystal Guo and Marcus Michelen for useful discussions regarding the skew-adjacency matrix $A_D$, Lutz Warnke for discussions regarding switching arguments, and Joel Spencer for discussions about Question~\ref{quest:RPS}.

\bibliographystyle{abbrv}
\bibliography{bib}
\newpage 
\appendix
\section{Proof of Theorem~\ref{thm:3path}}\label{sec:path}
This section is dedicated to proving Theorem~\ref{thm:3path}, which we recall says that if $D$ is the directed 3-path $1\to 2\to 3$, then the optimal strategies for Rei in the semi-restricted $D$-game are exactly those $\R$ satisfying $\Pr(\R(\r)=3)=1/2$ whenever $\{3\}\subsetneq \supp(\r)$.

\begin{proof}[Proof of Theorem~\ref{thm:3path}]Let $\R$ be an arbitrary optimal strategy for Rei, and as before we argue by induction that $\Pr(\R(\r)=3)=1/2$ whenever $\{3\}\subsetneq \supp(\r)$. We fix an arbitrary $\r$ with $\{3\}\subsetneq \supp(\r)$ and write $\Pr(\R(\r)=u)=p_u$ for each $u\in \{1,2,3\}$.
	
	If $p_2=0$, then by applying the techniques in the proof of Theorem~\ref{thm:RPS2}, we can argue $p_1=p_3=1/2$ as desired. Similarly if $p_1=0$, we can argue $p_2=p_3=1/2$. As such we skip these arguments and consider only the case  $p_1,p_2>0$.  Our aim is to show $p_2=p_3-p_1$, which implies $p_3=1/2$ as desired.
	
	If $p_2>p_3-p_1$, let
	$$\epsilon=\min \{p_2,p_2-(p_3-p_1)\}>0$$
	and consider a different strategy $\R'$ for Rei with
	$$\Pr(\R'(\r)=1)=p_1+\frac{\epsilon}{2},\ \Pr(\R'(\r)=2)=p_2-\epsilon,\ \Pr(\R'(\r)=3)=p_3+\frac{\epsilon}{2},$$
	and $\Pr(\R'(\r')=u)=\Pr(\R(\r')=u)$ for all vertices $u$ and $\r'\ne \r$. Note that, by our choice of $\epsilon$, we have
	$$\max\left\{p_2-\epsilon,\left(p_1+\frac{\epsilon}{2}\right)-\left(p_3+\frac{\epsilon}{2}\right),-(p_2-\epsilon)\right\}=\max\{p_2,p_3-p_1,-p_2\}-\epsilon.$$
	Using Fact~\ref{fact:recurrence} and that $\R$ is an optimal strategy, we find
	$$\S_D(\r;\R')-\S_D(\r;\R)=\epsilon\left(\frac{1}{2}\S_D(\r-\del_1)+\frac{1}{2}\S_D(\r-\del_3)-\S_D(\r-\del_2)-1\right).$$
	By Lemma~\ref{lem:switch} and the structure of $D$, we have
	$$\S_D(\r-\del_1)\leq \S_D(\r-\del_2)+1\text{ and }\S_D(\r-\del_3)<\S_D(\r-\del_2)+1,$$
	which implies $\S_D(\r;\R')-\S_D(\r;\R)<0$. This contradicts $\R$ being optimal, so we can not have $p_2>p_3-p_1$. If $p_2<p_3-p_1$, we can take $\epsilon=(p_3-p_1)-p_2>0$ and consider another strategy $\R'$ for Rei with
	$$\Pr(\R'(\r)=1)=p_1,\ \Pr(\R'(\r)=2)=p_2+\frac{\epsilon}{2},\ \Pr(\R'(\r)=3)=p_3-\frac{\epsilon}{2},$$
	with $\R'$ coinciding with $\R$ on every other vector $\r'$. Similarly, $\R'$ will contradict the fact that $\R$ is optimal, so we must have $p_2=p_3-p_1$ as wanted.  This proves that every optimal strategy $\R$ satisfies $\Pr(\R(\r)=3)=1/2$ when $\{3\}\subsetneq \supp(\r)$.  It remains to show that all such strategies are optimal, i.e.\ that $\S_D(\r;\R)=\S_D(\r;\R')$ for any strategies $\R,\R'$ of this form.
	
	Take $\R$ to be the unique strategy for Rei defined by $\Pr(\R(\r)=3)=1/2$ whenever $\{3\}\subsetneq \supp(\r)$ and $\Pr(\R(\r)=2)=0$ whenever $1\in \supp(\r)$. We shall prove by induction on $\r_1+\r_2+\r_3$ that\begin{itemize}
		\item $\S_D(\r;\R)=\S_D(\r;\R')$ for an arbitrary strategy $\R'$ satisfying $\Pr(\R'(\r)=3)=1/2$ whenever $\{3\}\subsetneq \supp(\r)$;
		\item $\S_D(\r)=\S_D(\r-\del_1+\del_2)-1$ whenever $\r_1>0$.
	\end{itemize}
	The base case when $\r_1+\r_2+\r_3=0$ is trivial. Suppose we have proved these two identities for all $\r'$ with $\r'_1+\r'_2+\r'_3< \r_1+\r_2+\r_3$.
	
	First we argue $\S_D(\r;\R)=\S_D(\r;\R')$. Let us write $\Pr(\R(\r)=u)=p_{u}$ and $\Pr(\R'(\r)=u)=p'_{u}$ for all $u$. If either $\r_1=0$ or $\r_2=0$, by the description of $\R$ and $\R'$, we know they coincide on all vectors $\r'$ with $\r'_1\leq \r_1$, $\r'_2\leq \r_2$, and $\r'_3\leq\r_3$. So we clearly have the desired identity in this case. If $\r_1,\r_2>0$ and $\r_3=0$, by the description of $\R,\R'$, we have $p_{1}=1$, $p_2=p_{3}=p'_{3}=0$, and $p'_1=1-p'_2$. Using Fact~\ref{fact:recurrence} we can compute\begin{align*}
		\S_D(\r;\R)-\S_D(\r;\R')&=\S_D(\r-\del_1)-\left(p_2'+(1-p_2')\S_D(\r-\del_1)+p'_2\S_D(\r-\del_2)\right)\\
		&=-p_2'(1-\S_D(\r-\del_1)+\S_D(\r-\del_2))=0,
	\end{align*}where the last equality follows from the second identity from the inductive hypothesis. If $\r_1,\r_2,\r_3>0$, by the description of $\R,\R'$, we have $p_{1}=1/2$, $p_2=0$, $p_{3}=p'_{3}=1/2$, and $p'_1=1/2-p'_2$. Again, by Fact~\ref{fact:recurrence} we can compute
	$$\S_D(\r;\R)-\S_D(\r;\R')=-p'_{2}(1-\S_D(\r-\del_1)+\S_D(\r-\del_2))=0.$$
	So we conclude the first wanted identity for our inductive process.
	
	Next we argue $\S_D(\r)=\S_D(\r-\del_1+\del_2)-1$ under the condition $\r_1>0$. Note that we have already showed $\S_D(\r;\R)=\S_D(\r;\R')$, and similarly we have $\S_D(\r-\del_1+\del_2;\R)=\S_D(\r-\del_1+\del_2;\R')$. So by the first half of this proof, we have
	$$\S_D(\r)=\S_D(\r;\R)\text{\quad and\quad}\S_D(\r-\del_1+\del_2)=\S_D(\r-\del_1+\del_2;\R).$$
	
	Now, if $\r_3=0$, it is easy to check $\S_D(\r)=\r_2=\S_D(\r-\del_1+\del_2)-1$ as wanted. If $\r_3>0$ and $\r_1=1$, by Fact~\ref{fact:recurrence} and the description of $\R$, we have
	\begin{align*}
		&\S_D(\r)=\frac{1}{2}\S_D(\r-\del_1)+\frac{1}{2}\S_D(\r-\del_3),\\
		&\S_D(\r-\del_1+\del_2)=\frac{1}{2}+\frac{1}{2}\S_D(\r-\del_1)+\frac{1}{2}\S_D(\r-\del_1+\del_2-\del_3).
	\end{align*} Together with the inductive hypothesis $\S_D(\r-\del_3)=\S_D(\r-\del_1+\del_2-\del_3)-1$, we can conclude the wanted identity. If $\r_3>0$ and $\r_1>1$, by Fact~\ref{fact:recurrence} and the description of $\R$, we have
	\begin{align*}
		&\S_D(\r)=\frac{1}{2}\S_D(\r-\del_1)+\frac{1}{2}\S_D(\r-\del_3),\\
		&\S_D(\r-\del_1+\del_2)=\frac{1}{2}\S_D(\r-2\del_1+\del_2)+\frac{1}{2}\S_D(\r-\del_1+\del_2-\del_3).
	\end{align*} Again, applying the inductive hypothesis for $\S_D(\r-\del_1)$ and $\S_D(\r-\del_3)$ respectively, we can conclude $\S_D(\r)=\r_2=\S_D(\r-\del_1+\del_2)-1$ as wanted. Hence we conclude the inductive process, which completes the proof.
\end{proof}
\newpage
\section{Proof of Theorem~\ref{thm:randomEulerian}}\label{sec:randomEulerian}
This section is dedicated to proving Theorem~\ref{thm:randomEulerian}, which we recall says that if $\c{E}_n$ is the set of Eulerian tournaments on $[2n+1]$ and $D$ is chosen uniformly at random from $\c{E}_n$, then
\[\Pr(D\textrm{ is oblivious})=O(n^{-2}).\]
The main lemma we need is the following.

\begin{lemma}\label{lem:switching}
	Let $\c{E}_{n}$ denote the set of Eulerian tournaments on $[2n+1]$, and for $v,w\in [2n+1]$, let $\c{E}_n^{v,w}\subseteq \c{E}_{n}$ denote the set of Eulerian tournaments $D$ with $vw\in E(D)$ and with $|N^+(v)\cap N^+(w)|=n-1$.  If $n\ge 16$, then
	\[|\c{E}_n^{v,w}|\le 2^{14}n^{-4} |\c{E}_n|.\]
\end{lemma}
\begin{proof}
	It is not difficult to see that for each $D\in \c{E}_n^{v,w}$, there exists a vertex $u_D$ and disjoint sets $A_D,B_D$ of size $n-1$ such that $N^+(v)\cap N^+(w)=A_D,\ N^-(v)\cap N^-(w)=B_D$, and $v\to w\to u_D\to v$.  We use this structural result together with a ``switching'' type argument to prove the bound.
	
	Given $D\in \c{E}_n^{v,w}$, we say that a pair of arcs $\{a_1b_1, a_2b_2\}$ is \textit{$D$-valid} if these are vertex disjoint arcs in $D$ with $a_1,a_2\in A_D$ and $b_1,b_2\in B_D$.  Given a $D$-valid pair $\{a_1b_1, a_2b_2\}$, we define $D[a_1b_1,a_2b_2]$ to be the digraph which has the arcs $b_1\to a_1\to v\to b_1$ and $b_2\to a_2\to v\to b_2$ and which otherwise agrees with $D$.  It is not difficult to see that $D[a_1b_1,a_2b_2]$ is an Eulerian tournament because $D$ was Eulerian and $\{a_1b_1, a_2b_2\}$ was $D$-valid.
	
	With the above definition in mind, we construct an auxiliary bipartite graph $G$ with vertex set $\c{E}_n^{v,w}\sqcup \c{E}_n$ by having $D\in \c{E}_n^{v,w}$ and $D'\in \c{E}_n$ adjacent in $G$ if and only if $D'=D[a_1b_1,a_2b_2]$ for some $D$-valid pair $\{a_1b_1,a_2b_2\}$.
	\begin{claim}
		Each $D\in \c{E}_n^{v,w}$ has $\deg_G(D)\ge 2^{-7} n^4$.
	\end{claim}
	\begin{proof}
		Because $D[a_1b_1,a_2b_2]$ is a distinct digraph for each distinct $D$-valid pair $\{a_1b_1,a_2b_2\}$, the claim is equivalent to saying each $D\in \c{E}_n^{v,w}$ has at least $2^{-7} n^2$ many $D$-valid pairs $\{a_1b_1,a_2b_2\}$.
		
		Given two (possibly non-disjoint) sets $S,T\subseteq V(D)$, let $e(S,T)$ be the number of arcs $s\to t$ in $D$ with $s\in S$ and $t\in T$.  Because $D$ is Eulerian and $|A_D|=n-1$, we have
		\begin{align*}n(n-1)=e(A_D,V(D))&=e(A_D,A_D)+e(A_D,\{v,w\})+e(A_D,\{u_D\})+e(A_D,B_D)\\&\le \binom{n-1}{2}+0+(n-1)+e(A_D,B_D),\end{align*}
		which implies $e(A_D,B_D)\ge \frac{1}{2}(n-1)(n-4)$.  Thus there exist at least $\frac{1}{2}(n-1)(n-4)$ choices for arcs $a_1 b_1$ from $A_D$ to $B_D$, and given such an arc, there exist at least \[e(A_D,B_D)-(n-1)\cdot 2\ge \frac{1}{2}(n-1)(n-8)\] arcs $a_2b_2$ with $a_1\ne a_2,b_1\ne b_2$.  In total then the number of ways we can construct an (ordered) $D$-valid pair $(a_1b_1,a_2b_2)$ is at least \[\half (n-1)(n-4)\cdot \half (n-1)(n-8)\ge 2^{-6} n^4\]
		for $n\ge 16$.  This double counts the number of (unordered) $D$-valid pairs $\{a_1b_1,a_2b_2\}$, so dividing this quantity by 2 gives the desired lower bound.
	\end{proof}
	\begin{claim}
		Each $D'\in \c{E}_n$ has $\deg_G(D')\le 120$.
	\end{claim}
	\begin{proof}
		Observe that if $D\in \c{E}_n^{v,w}$ and $\{a_1b_1,a_2b_2\}$ is $D$-valid, then there are exactly 5 vertices $x\in V(D[a_1b_1,a_2b_2])$ which have one arc to $\{v,w\}$ and one arc from $\{v,w\}$, namely this holds for $a_1, b_1, a_2, b_2,u_D$.  Thus $\deg_G(D')=0$ if $D'\in \c{E}_n$ does not have 5 vertices with this property, and otherwise there are trivially at most $5!=120$ digraphs $D$ which could have $D'=D[a_1b_1,a_2b_2]$ (namely by choosing which of its 5 vertices to play the roles of $a,a',b,b',u_D$).
	\end{proof}
	
	With these two claims, we have
	\[2^{-7} n^2|\c{E}_n^{v,w}|\le \sum_{D\in \c{E}_n^{v,w}} \deg_G(D)=\sum_{D'\in \c{E}_n} \deg_G(D')\le 120 |\c{E}_n|,\]
	and rearranging this inequality gives the result.
\end{proof}

We can now prove almost every Eulerian tournament is not oblivious.

\begin{proof}[Proof of Theorem~\ref{thm:randomEulerian}]
	Let $\c{E}_n^*\sub \c{E}_n$ be the set of oblivious Eulerian digraphs on $[2n+1]$. Then the theorem statement is equivalent to saying
	\[|\c{E}_n^*|=O(n^{-2} |\c{E}_n|).\]
	Fix an arbitrary set $S\sub [2n+1]$ of $2n$ vertices, and let $\c{E}_n^{**}$ be the set of digraphs which have $N^+(v)\cap S\subseteq N^+(w)\cap S$ for some distinct $v,w\in [2n+1]$.  We claim that \[\c{E}_n^*\sub \c{E}_n^{**}\sub \bigcup_{v,w} \c{E}_n^{v,w}.\]  
	
	Indeed, $\c{E}_n^*\sub \c{E}_n^{**}$ follows from Proposition~\ref{prop:oblivious}.  Let $D\in \c{E}_n^{**}$, say with $N^+(v)\cap S\subseteq N^+(w)\cap S$.  Note that $|N^+(v)\cap S|\ge |N^+(v)|-1=n-1$, so in particular $|N^+(v)\cap N^+(w)|\ge n-1$.  We can not have $N^+(v)=N^+(w)$ since either $vw\in E(D)$ or $wv\in E(D)$, so we must have $|N^+(v)\cap N^+(w)|= n-1$.  Thus $D\in \c{E}_n^{v,w}\cup \c{E}_n^{w,v}$, proving the claim.
	
	With this claim, we have
	\[|\c{E}_n^*|\le |\c{E}_n^{**}|\le \left|\bigcup_{v,w} \c{E}_n^{v,w}\right|=O(n^{-2} |\c{E}_n|),\]
	where this last step used Lemma~\ref{lem:switching}.  We conclude the result.
\end{proof}
We note that it is likely that a more sophisticated argument could be used to improve the bound of Theorem~\ref{thm:randomEulerian} to show that it is exponentially unlikely for a random Eulerian tournament to be oblivious.

\newpage 
\section{Restricted Games}\label{sec:restricted}
This section is devoted to studying games where both players are restricted.  In particular, we look at the case when two players Alice and Bob are given restriction vectors $\a,\b$ with $\sum \a_i=\sum \b_j=N$ and play some zero sum game a total of $N$ times, with e.g.\ Alice being forced to use each option $i$ a total of $\a_i$ times.  We begin with some formal definitions analogous to the definitions of Section~\ref{sec:optimal}. 

We use $G$ to denote a simultaneous zero-sum game played by Alice and Bob where $O_A$ ($O_B$ resp.) consists of the set of options Alice (Bob resp.) can play at each round.  We let $G(i,j)$ denote the score that Alice receives in this game if Alice plays $i$ and Bob plays $j$.  We say that $(\a,\b)$ is a pair of restriction vectors if $\a,\b$ are non-negative integral vectors indexed by $O_A,O_B$ with $\sum \a_i=\sum \b_j$.

Given a game $G$ as above, a \textit{strategy for Alice in the restricted $G$-game} is a function $\A$ from pairs of restriction vectors $(\a,\b)$ to random variables such that $\supp(\A(\a,\b))\subseteq \supp(\a)$ for all $(\a,\b)$. A \emph{strategy for Bob} $\B$ is defined analogously by requiring $\supp(\B(\a,\b))\subseteq \supp(\b)$. Recall that $\del_i$ is a restriction vector with value 1 on entry $i$ and 0 elsewhere.  Given a pair of strategies $\c{A},\c{B}$ for Alice and Bob and restriction vectors $(\a,\b)$, we define the \textit{score} (for Alice) $\mathbf{S}_G(\a,\b;\c{A},\c{B})$ by $\mathbf{S}_G(\a,\b;\c{A},\c{B})=0$ if $\sum\a_i=\sum\b_j=0$, and otherwise
\[\mathbf{S}_G(\a,\b;\c{A},\c{B})=G(\c{A}(\a,\b),\c{B}(\a,\b))+\mathbf{S}_G(\a-\delta_{\c{A}(\a,\b)},\b-\delta_{\c{B}(\a,\b)};\c{A},\c{B}).\]
We define the \textit{expected score} $\S_G(\a,\b;\A,\B)=\E[\mathbf{S}_G(\a,\b;\c{A},\c{B})]$.  We have the following analogue of Fact~\ref{fact:recurrence} for restricted games.
\begin{fact}\label{fact:restricted}
	Let $\A$ and $\B$ be strategies for Alice and Bob respectively. If $(\a,\b)$ is a restriction vector pair with $p_i=\Pr(\A(\a,\b)=i)$ and $q_j=\Pr(\B(\a,\b)=j)$ for each $i\in O_A,j\in O_B$, then
	\[\S_G(\a,\b;\A,\B)=\sum_i\sum_j p_iq_j(G(i,j)+\S_G(\a-\del_i,\b-\del_j;\A,\B)).\]
\end{fact}

A strategy for Alice $\c{A}$ is \textit{optimal} if for every strategy $\c{B}$ for Bob we have $S_G(\a,\b;\c{A},\c{B})=\max_{\c{A}'}S_G(\a,\b;\c{A}',\c{B})$, and we similarly define what it means for $\c{B}$ to be an optimal strategy for Bob.  We define the \textit{uniform strategy for Alice} $\c{U}_A$ by having $\Pr[\c{U}_A(\a,\b)=i]=\frac{\a_i}{\sum_{i'} \a_{i'}}$, and we similarly define the uniform strategy for Bob $\c{U}_B$.

\begin{proposition}\label{prop:restricted}
	In the restricted $G$-game, $\c{U}_A$ is an optimal strategy for Alice and $\c{U}_B$ is an optimal strategy for Bob.
\end{proposition}
\begin{proof}
	We will prove by induction on $N$ that if $\a,\b$ are restriction vectors with $\sum_i \a_i=\sum_j \b_j=N$ and if $\A$ is any strategy for Alice, then
	$$\S_G(\a,\b;\A,\U_B)=\sum_{i,j}\frac{G(i,j)\a_i\b_j}{N}=:M(\a,\b).$$
	The base case when $N=1$ is obvious. For the inductive process, suppose there exists $i\in O_A$ such that $\Pr(\A(\a,\b)=i)=1$. By Fact~\ref{fact:restricted}, we have
	\begin{equation}\S_G(\a,\b;\A,\U_B)=\sum_j \frac{\b_j}{N}\cdot (G(i,j)+\S_G(\a-\del_{i},\b-\del_j;\A,\U_B))\label{eq:restricted}.\end{equation}
	By the inductive hypothesis, we have
	\begin{align*}
		\S_G(\a-\del_{i},\b-\del_j;\A,\U_B)&=\frac{1}{N-1}\scalebox{2}{(} \sum_{i'\ne i,j'\ne j}G(i',j')\a_{i'}\b_{j'}+\sum_{i'\ne i}G(i',j)\a_{i'}(\b_j-1)\\
		&\hspace{1in}+\sum_{j'\ne j}G(i,j')(\a_i-1)\b_{j'}+G(i,j)(\a_i-1)(\b_j-1) \scalebox{2}{)}\\
		&=\frac{1}{N-1}\left(\sum_{i',j'}G(i',j')\a_{i'}\b_{j'}-\sum_{i'}G(i',j)\a_{i'}-\sum_{j'}G(i,j')\b_{j'}+G(i,j)\right).
	\end{align*}
	Plugging this into \eqref{eq:restricted} and using $\sum_j c\b_j =cN$ for any constant $c$, we have
	\begin{align*}
		N(N-1)\cdot \S_G(\a,\b;\c{A},\c{U}_B)&=\sum_jb_j\scalebox{2}{(}(N-1)G(i,j)+\sum_{i',j'}G(i',j')\a_{i'}\b_{j'}\\
		&\hspace{1in}-\sum_{i'}G(i',j)\a_{i'}-\sum_{j'}G(i,j')\b_{j'}+G(i,j)\scalebox{2}{)}\\
		&=(N-1) \sum_{i',j'} G(i',j')\a_{i'}\b_{j'}.
	\end{align*}
	Dividing both sides by $N(N-1)$ gives $\S_G(\a,\b;\A,\U_B)=M(\a,\b)$ under our assumption $\Pr(\A(\a,\b)=i)=1$. Since this holds regardless of which deterministic option $i$ Alice plays, we see that $\S_G(\a,\b;\A,\U)=M(\a,\b)$ for any (mixed) strategy $\A$ for Alice, concluding the inductive proof.
	
	As we have shown, Bob can use strategy $\U$ to guarantee that Alice's expected score is at most $M(\a,\b)$. A symmetric argument shows that Alice can guarantee at least $M(\a,\b)$ points by using her uniform strategy. Thus $M(\a,\b)$ is the expected score if both players play optimally, and in particular $\c{U}_A$ and $\c{U}_B$ are optimal strategies for Alice and Bob respectively.
\end{proof}
We note that in general $\c{U}_A$ and $\c{U}_B$ will not be the unique optimal strategies for the restricted $G$-game. Indeed, if $G(i,j)=0$ for all $i,j$ then every strategy is optimal, and more generally there will be multiple optimal strategies if there exist $i\ne i'$ such that $G(i,j)=G(i',j)$ for all $j$.
\end{document}